\documentclass[10pt]{amsart}

\input{GG1.sty}

\title[Homotopy types of $\SU(n)$-gauge groups over spheres]{Lower bounds on the number of homotopy types of $\SU(n)$-gauge groups over spheres}
\author{Martin Frankland and Yang Hu}
\address{Department of Mathematics and Statistics, University of Regina}
\email{Martin.Frankland@uregina.ca}
\address{Department of Mathematics and Statistics, University of Regina}
\email{Yang.Hu@uregina.ca} 
\date{}         


\begin{document}

\begin{abstract}
We explore the classical question of enumerating the homotopy types of gauge groups of principal $\SU(n)$-bundles over $\Smr^{2m+1}$ in the unstable range $m\geq n$. Using computations of certain $p$-local Samelson products, we obtain lower bounds on the number of such homotopy types. We also present a discussion of the last stable case.
\end{abstract}

\maketitle

\tableofcontents

\section{Introduction} \label{sec:intro}

Gauge groups are fundamental structures studied across mathematics and physics. 
Given a compact connected Lie group $\G$, and a principal $\G$-bundle $\sfP$ over a connected finite CW complex $X$, the gauge group $\Gcal_{\sfP}$ of $\sfP$ is defined as the group of $\G$-equivariant bundle automorphisms. 

\begin{q} \label[q]{main_q}
How many homotopy types of $\Gcal_{\sfP}$ are there, as $\sfP$ ranges over all principal $\G$-bundles over a fixed base complex $X$?
\end{q}

The first step towards answering \Cref{main_q} is to enumerate the principal bundles. That is, to compute
\[
\Prin_{\G}(X) \cong [X, \BG],
\]
where $\Prin_{\G}(X)$ denotes the set of isomorphism classes of principal $\G$-bundles over $X$, before enumerating the homotopy types of $\Gcal_{\sfP}$, $\sfP\in \Prin_{\G}(X)$.
In general, taking $\Gcal_{(-)}$ compresses information, as non-isomorphic bundles can give rise to homotopy equivalent gauge groups. 
Indeed, Crabb and Sutherland \cite{CS00} prove that the answer to \Cref{main_q} is always a finite number, 
although $\Prin_{\G}(X)$ can be an infinite set.

\smallskip

This general discrepancy between the two enumeration questions is well demonstrated in the cases extensively studied in literature, namely when the group $\G$ is a classical Lie group $\SU(n)$ or $\text{Sp}(n)$, and $X$ is a 4-manifold (see \cite{So19} for a list of such results) or a higher-dimensional sphere.
For example, for the principal $\SU(3)$-bundles over $\Smr^6$, their associated rank $3$ complex vector bundles are in the stable range, and are therefore classified as:
\[
\Prin_{\SU(3)}(\Smr^6) \, \cong \, [\Smr^6, \BSU(3)] \, \cong \, \KUt^0(\Smr^6) \, \cong \, \ZZ.
\]
Furthermore, the third Chern class $c_3\in 2\ZZ \subset \ZZ \cong \Hmr^6(\Smr^6; \ZZ)$ completely determines these bundles. 
Denote by $\Gcal_k$ the gauge group of the bundle $\sfP$ with $c_3(\sfP) = 2k$, $k\in \ZZ$. Hamanaka and Kono \cite{HK07} prove that $\Gcal_k \simeq \Gcal_{k'}$ precisely when $\gcd(120, k) = \gcd(120, k')$. So there are exactly $16$ distinct homotopy types of gauge groups, although there exist a $\ZZ$-worth of principal bundles.

\smallskip

For general $\SU(n)$-gauge groups over spheres $\Smr^r$, obtaining a complete answer to \Cref{main_q} becomes significantly harder as $n$ and $r$ become larger. When $r \leq 2n$,  rank $n$ complex vector bundles over $\Smr^r$ are stable and are uniquely determined by the Chern class $c_n$, which is of the form $(n-1)! \cdot k$, $k\in \ZZ$. In these cases, some lower bounds on the number of homotopy types of gauge groups are present. For example, for $\SU(4)$-gauge groups over $\Smr^8$, Mohammadi and Asadi-Golmankhaneh \cite{MAG19} prove that if $\Gcal_k\simeq \Gcal_{k'}$ then $\gcd(420, k) = \gcd(420, k')$. In contrast, very little information is known in the unstable range $r>2n$.

\smallskip

This paper is aimed at producing systematic lower bounds for the number of homotopy types of $\SU(n)$-gauge groups over general spheres $\Smr^{r}$, both in the last stable case $r=2n$ (\Cref{sec:sun}) and in a number of unstable cases $r=2m+1$, $m\geq n$ (\Cref{sec:unstable}). 

\smallskip

In the last stable case $r=2n$, we follow the classical methods.
Write $\sfP_k$ for the principal $\SU(n)$-bundle over $\Smr^{2n}$ with $c_n = (n-1)!\cdot k$, $k\in \ZZ$. Its classifying map is $k\epsilon: \Smr^{2n} \rightarrow \BSU(n)$, where $\epsilon$ denotes the standard generator of $\pi_{2n}\BSU(n) \cong \pi_{2n-1}\SU(n)\cong \ZZ$. Write $\G_k$ for the gauge group of $\sfP_k$.
By the work of Gottlieb \cite{Go72}, or that of Atiyah and Bott \cite{AB83}, the classifying space $\B\Gcal_k$ fits into a fiber sequence
\[
\begin{tikzcd}
\Gcal_k \ar[r] & \SU(n) \ar[r, "\partial_k"'] & \Omega_0^{2n-1}\SU(n) \ar[r] & \B\Gcal_k \ar[r] & \BSU(n).
\end{tikzcd}
\]
By a result of Lang \cite{La73}, the map $\partial_k$ in the sequence is adjoint to the Samelson product
\[
\langle k\epsilon, 1_{\SU(n)} \rangle:
\begin{tikzcd}
 \Smr^{2n-1} \sma \SU(n) \ar[rr, "k\epsilon \sma 1_{\SU(n)}"'] & &  \SU(n)\sma \SU(n) \ar[r, "c"'] & \SU(n)
\end{tikzcd}
\] 
where $c$ denotes the map taking the commutator. Using methods developed by Hamanaka and Kono \cite{HK07}, in \Cref{sec:sun} we prove:
\begin{thm}[See \Cref{thm:lowerboundnodd} and \Cref{thm:lowerboundneven}] \label{thm:mainstable}
If $n$ is odd, then $\Gcal_k\simeq \Gcal_{\ell}$ implies that
\[
\gcd(2n(n+1)(n+2), \, k) = \gcd(2n(n+1)(n+2), \, \ell).
\]
If $n$ is even, then $\Gcal_k\simeq \Gcal_{\ell}$ implies that
\[
\begin{cases}
\gcd(n(n+1)(n+2)/2, \, k) = \gcd(n(n+1)(n+2)/2, \, \ell), & \text{ if } n \equiv 0 \text{ mod } 4 \\
\gcd(n(n+1)(n+2), \, k) = \gcd(n(n+1)(n+2), \, \ell), & \text{ if } n \equiv 2 \text{ mod } 4.
\end{cases}
\]
\end{thm}
The same lower bound in the cases $n \equiv 0, 1, 3$ mod $4$ was obtained by S.\ Mohammadi \cite{Moh22}, using similar Hamanaka--Kono methods. However, \Cref{thm:mainstable} gives an improved lower bound in the case $n \equiv 2$ mod $4$.

\smallskip

The unstable situation is drastically different. For an $\SU(n)$-bundle over $\Smr^{r}$, $r > 2n$, the associated complex vector bundle is unstable. In particular,
\[
\Prin_{\SU(n)}(\Smr^{r}) \cong \pi_{r}\BSU(n) \cong \pi_{r-1}\SU(n),
\]
which is a {\emph{finite}} abelian group. Given that taking $\Gcal_{(-)}$ often compresses lots of information,  it is natural to wonder if these finitely many principal bundles become homotopically indistinguishable after passing to gauge groups.

\begin{q} \label[q]{q2}
Does the homotopy type of 
the gauge group of 
a non-trivial $\SU(n)$-bundle over $\Smr^{r}$, $r > 2n$, always coincide with that of the trivial bundle?
\end{q}

In \Cref{sec:unstable} we address \Cref{q2} for unstable $\SU(n)$-bundles over odd spheres $\Smr^{2m+1}$, $m\geq n$, and prove that the answer is generally no. Moreover, we prove:

\begin{thm}[See \Cref{count} and \Cref{thm:counting}] \label{thm:main02}
For a pair of integers $(n, j)$ with $n\geq 3$ and $j\geq 0$, denote by $\theta(n, j)$ the number of odd primes satisfying that $p\geq j+2$, that $p\leq n+1$, and that the value of $n$ mod $p$ belongs to the set $\{0, p-1, \cdots. p-j\}$.
There are at least $2^{\theta(n, j)}$ distinct homotopy types of gauge groups $\Gcal_{\sfP}$, as $\sfP$ ranges over the principal $\SU(n)$-bundles over $\Smr^{2n+2j+1}$.
\end{thm}

For example, taking $j=0$ and $j=1$ respectively and applying \Cref{thm:main02}, we obtain the following systematic lower bounds for the number of homotopy type of $\SU(n)$-gauge groups.

\begin{cor}[See \Cref{corank1}]
The number of homotopy types of $\SU(n)$-gauge groups over $\Smr^{2n+1}$ is at least
\[
 \begin{cases}
 2^{\omega(n)} & \text{if } n \text{ is odd} \\
 2^{\omega(n) - 1} & \text{if } n \text{ is even},
\end{cases}
\]
where $\omega(n)$ denotes the number of primes dividing $n$.
\end{cor}

\begin{cor}[See \Cref{corank2}]
The number of homotopy types of $\SU(n)$-gauge groups over $\Smr^{2n+3}$ is at least \[2^{\omega(n) + \omega(n+1) - 1}.\]
\end{cor}

More concrete examples can be obtained by setting specific values of $n$ and $j$ and applying \Cref{thm:main02}. We present one example below, and more can be found in \Cref{subsec:application}.

\begin{ex}
Consider the principal $\SU(1012)$-bundles over $\Smr^{2027}$. We have:
\[
\Prin_{\SU(1012)}(\Smr^{2027}) \cong \pi_{2026}\SU(1012) \cong \ZZ/1013!\oplus \ZZ/2.
\]
So up to isomorphism there are $2\cdot (1013!)$ such bundles. Applying \Cref{thm:main02} with $n=1012$ and $j=1$, we find that
\[
\theta(1012, 1) = 3.
\]
So the $2\cdot (1013!)$ bundles give rise to at least $2^3 = 8$ distinct homotopy types of gauge groups.
Similarly, we find that
\[
\theta(1003, 10) = 10, \quad \text{and} \quad \theta(838, 175) = 40.
\]
So \Cref{thm:main02} implies that there are at least $2^{10}$ distinct homotopy types of $\SU(1003)$-gauge groups over $\Smr^{2027}$, and at least  $2^{40}$ distinct homotopy types of $\SU(838)$-gauge groups over $\Smr^{2027}$.
\end{ex}

The key to proving \Cref{thm:counting} is the calculation of a Samelson product (see \Cref{thm:Sam}). This is a generalized version of the calculation conducted in \cite{Ha07}. The details of the Samelson product calculation are presented in \Cref{subsec:Samelson}. In \Cref{subsec:application} we apply the calculation to obtain lower bounds on the number of homotopy types of the gauge groups in question.

\subsection*{Acknowledgment}
The authors thank Larry So for helpful conversations. 
Frankland 
acknowledges the support of the Natural Sciences and Engineering Research Council of Canada (NSERC), grant RGPIN-2026-07853, as well as the support of the Max-Planck-Institut für Mathematik, report number MPIM-Bonn-2026. 
Hu 
acknowledges the support of the Pacific Institute for the Mathematical Sciences (PIMS) through CRG41. 


\section{Homotopy types of \texorpdfstring{$\SU(n)$}{SUn}-gauge groups over \texorpdfstring{$\Smr^{2n}$}{S2n}} \label{sec:sun}

This section is dedicated to the last stable case.
Namely, we detect homotopy types of $\SU(n)$-gauge groups over $\Smr^{2n}$ for a general $n$. 
The starting point is the fiber sequence:
\begin{equation} \label{eq:funfibn}
\begin{tikzcd}
\Gcal_k \ar[r] & \SU(n) \ar[r, "\partial_k"'] & \Omega_0^{2n-1}\SU(n) \ar[r] & \B\Gcal_k \ar[r] & \BSU(n)
\end{tikzcd}
\end{equation}
where the connecting map
\[
\partial_k \in [\SU(n), \Omega_0^{2n-1}\SU(n)] \cong [\SU(n), \Omega^{2n-1}\SU(n)] \cong [\Sigma^{2n-1}\SU(n), \SU(n)]
\]
is the Samelson product
\[
\langle k\epsilon, 1_{\SU(n)} \rangle:
\begin{tikzcd}
 \Smr^{2n-1} \sma \SU(n) \ar[rr, "k\epsilon \sma 1_{\SU(n)}"'] & &  \SU(n)\sma \SU(n) \ar[r, "c"'] & \SU(n).
\end{tikzcd}
\]

We use the finite complex $\Sigma\CP^2$ as the test complex.
Applying the functor $[\Sigma\CP^2, -]$ to \eqref{eq:funfibn} yields an exact sequence:
\begin{equation} \label{eq:lesn}
\begin{tikzcd}[column sep= small]
{[\Sigma\CP^2, \Gcal_k]} \ar[r] & {[\Sigma\CP^2, \SU(n)]} \ar[r, "(\partial_k)_*"'] & {[\Sigma\CP^2, \Omega_0^{2n-1}\SU(n)]} \ar[r] & {[\Sigma\CP^2, \B\Gcal_k]} \ar[r] & 0.
\end{tikzcd}
\end{equation}
The last group in \eqref{eq:lesn} is the zero group, since for $n > 2$, 
\[
[\Sigma\CP^2, \BSU(n)] \cong [\Sigma\CP^2, \BU] \cong \KUt^0(\Sigma\CP^2) = 0.
\]
It follows that:
\begin{align*}
\Gcal_k \simeq \Gcal_{\ell} & \implies [\Sigma\CP^2, \B\Gcal_k] \cong[\Sigma\CP^2, \B\Gcal_{\ell}] \\
& \implies |\Im (\partial_k)_*| = |\Im (\partial_\ell)_*|.
\end{align*}
We calculate $\Im (\partial_k)_*$ to obtain necessary conditions for $\Gcal_k\simeq \Gcal_{\ell}$.
Our main results are:

\begin{thm} \label{thm:lowerboundnodd}
Suppose that $n$ is odd. In this case, if $\Gcal_k\simeq \Gcal_{\ell}$ then
\[
\gcd(2n(n+1)(n+2), \, k) = \gcd(2n(n+1)(n+2), \, \ell).
\]
\end{thm}

\begin{thm} \label{thm:lowerboundneven}
Suppose that $n$ is even. In this case, if $\Gcal_k\simeq \Gcal_{\ell}$ then
\[
\begin{cases}
\gcd(n(n+1)(n+2)/2, \, k) = \gcd(n(n+1)(n+2)/2, \, \ell), & \text{ if } n \equiv 0 \text{ mod } 4 \\
\gcd(n(n+1)(n+2), \, k) = \gcd(n(n+1)(n+2), \, \ell), & \text{ if } n \equiv 2 \text{ mod } 4.
\end{cases}
\]
\end{thm}

\smallskip

In what follows, we treat the abelian group
\[
[\Sigma\CP^2, \Omega_0^{2n-1}\SU(n)] \cong [\Sigma^{2n}\CP^2, \SU(n)]
\]
as a subgroup of $[\Sigma^{2n}\CP^2, \U(n+1)]$. Indeed, writing $i$ for the standard subgroup inclusion $\SU(n) \rightarrow \U(n+1)$, we have:

\begin{prop} \label[prop]{subpresentn}
The induced homomorphism
\[
i_*: [\Sigma^{2n}\CP^2, \SU(n)] \longrightarrow [\Sigma^{2n}\CP^2, \U(n+1)]
\]
is injective.
\end{prop}

\begin{proof}
Since $\Sigma^{2n}\CP^2$ is simply-connected, the induced map 
\[
[\Sigma^{2n}\CP^2, \SU(n+1)] \to [\Sigma^{2n}\CP^2, \U(n+1)]
\]
is injective. Hence it suffices to show that
\[
i_*: [\Sigma^{2n}\CP^2, \SU(n)] \longrightarrow [\Sigma^{2n}\CP^2, \SU(n+1)]
\]
is injective. Since there is the fiber sequence
\[
\Omega\Smr^{2n+1} \longrightarrow \SU(n) \longrightarrow \SU(n+1) \longrightarrow \Smr^{2n+1},
\]
it suffices to prove that
\[
[\Sigma^{2n}\CP^2, \Omega\Smr^{2n+1}] = [\Sigma^{2n+1}\CP^2, \Smr^{2n+1}] = 0.
\]
%

Consider the defining cofiber sequence of $\Sigma^{2n+1}\CP^2$:
\[
\Smr^{2n+4} \stackrel{\eta}{\longrightarrow} \Smr^{2n+3} \longrightarrow \Sigma^{2n+1}\CP^2 \longrightarrow \Smr^{2n+5} \stackrel{\eta}{\longrightarrow} \Smr^{2n+4}.
\]
Applying the functor $[-, \Smr^{2n+1}]$ gives rise to an exact sequence:
\[
0 \longrightarrow [\Sigma^{2n+1}\CP^2, \Smr^{2n+1}] \longrightarrow [\Smr^{2n+3}, \Smr^{2n+1}]  \stackrel{\eta^*}{\longrightarrow} [\Smr^{2n+4}, \Smr^{2n+1}].
\]
The first group in the exact sequence is zero, since $[\Smr^{2n+5}, \Smr^{2n+1}] = \pi_4^{s} = 0$ (for $n \geq 3$). 
Also since 
\begin{enumerate}
 \item[-] $[\Smr^{2n+3}, \Smr^{2n+1}] = \pi_2^{s} \cong \ZZ/2 \cdot \eta^2$, and
 \item[-] $[\Smr^{2n+4}, \Smr^{2n+1}] = \pi_3^{s} \cong \ZZ/24$, which is 2-locally $\ZZ/8\cdot \nu$ and 3-locally $\ZZ/3 \cdot \alpha_1$,
\end{enumerate}
the homomorphism $\eta^*$ in the exact sequence takes the generator $\eta^2$ to $\eta^3 = 4\nu$, and is therefore injective. 
This forces $[\Sigma^{2n+1}\CP^2, \Smr^{2n+1}]$ to be the zero group.
\end{proof}

The advantage of working inside of $\U(n+1)$ is that it allows one to detect $\Im(\partial_k)_*$ by \mbox{$\Hmr\ZZ$-cohomology}. Indeed, writing $\V_{n+1}$ for the infinite Stiefel manifold $\U/\U(n+1)$, and $\pi$ for the standard quotient map $\U \rightarrow \V_{n+1}$, one has the fiber sequence
\[
\Omega\U \stackrel{\Omega\pi}{\longrightarrow} \Omega\V_{n+1} \longrightarrow \U(n+1) \longrightarrow \U\stackrel{\pi}{\longrightarrow} \V_{n+1}.
\]
We recall the $\Hmr\ZZ$-cohomology of some of the spaces involved:
\begin{itemize}
 \item[-] $\Hmr^*(\U; \ZZ)$ is the exterior algebra $\Lambda(x_1, x_3, x_5, \cdots)$, where $x_{2i+1}$ is of degree $2i+1$  and transgresses to the $(i+1)$-st universal Chern class $c_{i+1} \in \Hmr^{2i+2}(\BU; \ZZ)$.
 \item[-] $\Hmr^*(\V_{n+1}; \ZZ)$ is the exterior algebra $\Lambda(\bar{x}_{2n+3}, \bar{x}_{2n+5}, \cdots)$, where $\bar{x}_{2i+1}$ is of degree $2i+1$ (for $i \geq n+1$) and is such that $\pi^*(\bar{x}_{2i+1})  = x_{2i+1} \in \Hmr^{2i}(\U; \ZZ)$.
 \item[-] $\Hmr^*(\Omega\V_{n+1}; \ZZ)$ is the symmetric algebra $\ZZ[a_{2n+2}, a_{2n+4}, \cdots]$, where $a_{2i}$ is of degree $2i$ (for $i\geq n+1$) and transgresses to $\bar{x}_{2i+1} \in \Hmr^{2i+1}(\V_{n+1}; \ZZ)$.
\end{itemize}
We need the following result, originally proved in \cite[Proposition 3.1]{HK03}.

\begin{prop} \label[prop]{omegapistar}
Let $\ch_n$ denote the degree $2n$ component of the universal Chern character. Then
\begin{equation} \label{eq:chch}
(\Omega\pi)^*(a_{2n}) = n! \cdot \ch_n
\end{equation}
in $\Hmr^{2n}(\BU;\ZZ)$.
\end{prop}

To compute $[\Sigma^{2n}\CP^2, \U(n+1)]$, we map $\Sigma^{2n}\CP^2$ to the diagram of spaces:
\[
\begin{tikzcd}
\Omega\U \ar[r, "\Omega\pi"] & \Omega\V_{n+1} \ar[r] \ar[d, "a_{2n+2}\times a_{2n+4}"] & \U(n+1) \ar[r] & \U \ar[r, "\pi"] & \V_{n+1} \\
 & \K(\ZZ, 2n+2) \times \K(\ZZ, 2n+4) & & & 
\end{tikzcd}
\]
to obtain the following diagram of abelian groups:
\begin{equation} \label{eq:testUnplus1}
\begin{tikzcd}[column sep = tiny]
 \KUt^{0}(\Sigma^{2n}\CP^2) \ar[r, "(\Omega\pi)_*"] & {[\Sigma^{2n}\CP^2, \Omega\V_{n+1}]} \ar[r] \ar[d, "\lambda:= (a_{2n+2}\times a_{2n+4})_*"] & {[\Sigma^{2n}\CP^2, \U(n+1)]} \ar[r] & 0 \\
 & \Hmr^{2n+2}(\Sigma^{2n}\CP^2; \ZZ) \oplus \Hmr^{2n+4}(\Sigma^{2n}\CP^2; \ZZ) & &  
\end{tikzcd}
\end{equation}
where the top row is exact. Note that the last group in the top row is zero, as 
$[\Sigma^{2n}\CP^2, \U] \cong \KUt^{1}(\Sigma^{2n}\CP^2) = 0$.
We need to find the image of the homomorphism $\lambda$. The result, and hence later results, would depend on the parity of $n$, and we therefore treat them separately.

\subsection{The case when \texorpdfstring{$n$}{n} is odd} \label{subsec:oddn}
We first deal with the case when $n$ is odd.  

\begin{prop} \label[prop]{lambdainjnodd}
The homomorphism $\lambda$ in the diagram \eqref{eq:testUnplus1} is injective, whose image is
\[
\{(a, b)\in \ZZ\oplus \ZZ: a \equiv 0 \text{ mod } 2\}.
\]
\end{prop}

\begin{proof}
Note that $\Sigma^{2n}\CP^2$ is $(2n+4)$-dimensional, and that the $(2n+4)$-skeleton of $\Omega\V_{n+1}$ is $\Smr^{2n+2}\vee \Smr^{2n+4}$. Indeed, the $(2n+4)$-cell is attached trivially, since it is not detected by $\Sq^2$. Since $\Omega\V_{n+1}$ has no $(2n+5)$-cell, it follows that 
\[
[\Sigma^{2n}\CP^2, \Omega\V_{n+1}] \cong [\Sigma^{2n}\CP^2, \Smr^{2n+2}\vee \Smr^{2n+4}] \cong [\Sigma^{2n}\CP^2, \Smr^{2n+2}] \oplus [\Sigma^{2n}\CP^2, \Smr^{2n+4}].
\]
Consider the defining cofiber sequence for $\Sigma^{2n}\CP^2$:
\[
\Smr^{2n+3} \stackrel{\eta}{\longrightarrow} \Smr^{2n+2} \longrightarrow \Sigma^{2n}\CP^2  \longrightarrow \Smr^{2n+4} \stackrel{\eta}{\longrightarrow} \Smr^{2n+3}.
\]
Applying the functor $[-, \Smr^{2n+4}]$ gives rise to the exact sequence:
\[
0 \longrightarrow [\Smr^{2n+4}, \Smr^{2n+4}] \longrightarrow [\Sigma^{2n}\CP^2, \Smr^{2n+4}]  \longrightarrow 0
\]
from which it follows that $[\Sigma^{2n}\CP^2, \Smr^{2n+4}] \cong \pi_{2n+4}(\Smr^{2n+4}) \cong \ZZ$. 
Applying the functor $[-, \Smr^{2n+2}]$, we obtain the following exact sequence:
\[
[\Smr^{2n+3}, \Smr^{2n+2}] \stackrel{\eta^*}{\longrightarrow} [\Smr^{2n+4}, \Smr^{2n+2}] \longrightarrow [\Sigma^{2n}\CP^2, \Smr^{2n+2}] \longrightarrow [\Smr^{2n+2}, \Smr^{2n+2}] \stackrel{\eta^*}{\longrightarrow} [\Smr^{2n+3}, \Smr^{2n+2}].
\]
The first $\eta^*$ is an isomorphism, since $\pi_{2n+3}(\Smr^{2n+2}) = \ZZ/2\cdot \eta$ and $\pi_{2n+4}(\Smr^{2n+2}) = \ZZ/2\cdot \eta^2$. The last $\eta^*$ is surjective with kernel $2\ZZ \subset \ZZ = \pi_{2n+2}(\Smr^{2n+2})$, since $\pi_{2n+2}(\Smr^{2n+2}) = \ZZ \cdot 1$ and $\pi_{2n+3}(\Smr^{2n+2}) = \ZZ/2\cdot \eta$.
It follows that $[\Sigma^{2n}\CP^2, \Smr^{2n+2}] \cong 2 \ZZ$. To sum up:
\[
[\Sigma^{2n}\CP^2, \Omega\V_{n+1}] \cong [\Sigma^{2n}\CP^2, \Smr^{2n+2}] \oplus [\Sigma^{2n}\CP^2, \Smr^{2n+4}] \cong 2\ZZ \oplus \ZZ,
\]
and the result is proved.
\end{proof}

Given \Cref{lambdainjnodd}, elements of $[\Sigma^{2n}\CP^2, \Omega\V_{n+1}]$ can be identified as a pair of integers $(a, b)$ where $a$ is even.

\begin{cor} \label[cor]{Unplus1basis}
The abelian group $[\Sigma^{2n}\CP^2, \Omega\V_{n+1}]$ is free of rank $2$, whose basis can be taken as
$\{(2, 0), (0, 1)\}$, or $\{(2, k), (0, 1)\}$  for any $k\in \ZZ$.
\end{cor}

In particular, $u:=(2, n+2)$ and $v:=(0, 1)$ form a basis for $[\Sigma^{2n}\CP^2, \Omega\V_{n+1}]$.
We now prove:

\begin{prop} \label[prop]{imomegapistarnodd}
The image of the homomorphism
\[
(\Omega\pi)_*: \KUt^{0}(\Sigma^{2n}\CP^2) \longrightarrow [\Sigma^{2n}\CP^2, \Omega\V_{n+1}]
\] 
is the subgroup of $[\Sigma^{2n}\CP^2, \Omega\V_{n+1}]$ generated by $\frac{(n+1)!}{2} u$ and $(n+2)! v$.
\end{prop}

\begin{proof}
The abelian group $\KUt^{0}(\Sigma^{2n}\CP^2)$ is free of rank $2$ generated by $\beta^nx$ and $\beta^nx^2$, where $\beta$ is the Bott element. We find that 
\[
\lambda \circ \Omega\pi \circ \beta^nx = (a_{2n+2} \circ \Omega\pi \circ \beta^n x, \, a_{2n+4} \circ \Omega\pi \circ \beta^n x), 
\]
which equals
\[
\left( (n+1)! \ch_{n+1}(\beta^n x), \, (n+2)! \ch_{n+2}(\beta^n x) \right)
\]
by the equation \eqref{eq:chch}. Denote by $\sigma$ the (cohomological) suspension isomorphism and by $t$ the generator of $\Hmr^2(\CP^2;\ZZ) \cong \ZZ$. Since $\ch_{n+1}(\beta^nx) = \sigma^{2n}t$ and $\ch_{n+2}(\beta^nx) = \sigma^{2n}t^2/2$, 
it follows that
\[
\lambda \circ \Omega\pi \circ \beta^n x = \left( (n+1)!, \, (n+2)!/2 \right).
\]
Similarly, using $\ch(x^2) = \ch(x)^2 = (t + \frac{t^2}{2})^2 = t^2 \in \Hmr^*(\CP^2;\QQ)$, 
one finds that
\[
\lambda \circ \Omega\pi \circ \beta^n x^2 = (0, \, (n+2)!).
\]
The result follows as $\left( (n+1)!, \, (n+2)!/2 \right) = \frac{(n+1)!}{2}u$ and $(0, (n+2)!) = (n+2)!v$.
\end{proof}

\begin{cor}
The abelian group $[\Sigma^{2n}\CP^2, \U(n+1)]$ is isomorphic to $\ZZ/\frac{(n+1)!}{2} \oplus \ZZ/(n+2)!$.
\end{cor}

\begin{proof}
Note that $[\Sigma^{2n}\CP^2, \U(n+1)]$ is the cokernel of $(\Omega\pi)_*$ by the diagram \eqref{eq:testUnplus1}, and that $\{u, v\}$ is a basis for $[\Sigma^{2n}\CP^2, \Omega\V_{n+1}]$ by \Cref{Unplus1basis}. Now \Cref{imomegapistarnodd} implies the result.
\end{proof}

We now return to the main goal of computing $\Im (\partial_k)_*$ from the exact sequence \eqref{eq:lesn}.
We first make a list of general observations. 

\begin{enumerate}[label=(\roman*)]
 \item The abelian group $[\Sigma\CP^2, \SU(n)]$ is in the stable range, and the exact sequence \eqref{eq:lesn} is therefore of the form
\[
\begin{tikzcd}
\KUt^0(\CP^2) \ar[r, "(\partial_k)_*"'] & {[\Sigma^{2n}\CP^2, \SU(n)]} \ar[r] & {[\Sigma\CP^2, \B\Gcal_k]} \ar[r] & 0
\end{tikzcd}
\]
where $(\partial_k)_*$ behaves as:
\[
\alpha\in \KUt^0(\CP^2) \longmapsto \langle k\epsilon, \alpha \rangle \in [\Sigma^{2n}\CP^2, \SU(n)].
\]
Since $\KUt^0(\CP^2)$ is additively generated by $x = \upgamma - 1$ and $x^2$, the image of $(\partial_k)_*$ is generated by $(\partial_k)_*(x) = \langle k\epsilon, \bar{x} \rangle = c\circ (k\epsilon \sma \bar{x})$ and $(\partial_k)_*(x^2) = \langle k\epsilon, \bar{x}' \rangle = c\circ (k\epsilon \sma \bar{x}')$, which are composites:
\[
\Sigma^{2n}\CP^2 = \Smr^{2n-1} \sma \Sigma\CP^2 \longrightarrow \SU(n)\sma \SU(n) \longrightarrow \SU(n).
\]
Here we write $\bar{x}: \Sigma\CP^2 \rightarrow \SU(n)$ for the adjoint of $x$, and $\bar{x}'$ for the adjoint of $x^2$.
Note that $\bar{x}$ is the standard inclusion and $\bar{x}'$ factors through the top cell.
\smallskip
\item Since $i_*$ is an injection by \Cref{subpresentn}, the size of 
\[
\Im (\partial_k)_* \subset [\Sigma^{2n}\CP^2, \SU(n)]
\]
is the same as the size of 
\[
\Im i_*\circ (\partial_k)_*  \subset [\Sigma^{2n}\CP^2, \U(n+1)].
\]
\item The linearity of the Samelson product implies that $\partial_k = k\partial_1$, so the computation can be reduced to the case $k=1$.
\end{enumerate}

\smallskip

Consider the following commutative diagram:
\begin{equation} \label{eq:commdiagn}
\begin{tikzcd}
 \SU(n) \sma \SU(n) \ar[r, "i\sma i"] \ar[d, "c"'] & \U(n+1) \sma \U(n+1) \ar[d, "c"] \ar[r, dashed, "\exists \tilde{c}"] & \Omega\V_{n+1}  \ar[dl, "j"] \\
 \SU(n) \ar[r, "i"'] & \U(n+1) & 
\end{tikzcd}
\end{equation}
where $i$ is the standard inclusion $\SU(n) \rightarrow \U(n+1)$, and $c$ denotes the map taking commutator. Note that the composite
\[
\U(n+1) \sma \U(n+1) \stackrel{c}{\longrightarrow} \U(n+1) \longrightarrow \U
\]
is null-homotopic since $\U$ is abelian. Therefore the commutator map $c: \U(n+1) \sma \U(n+1) \rightarrow \U(n+1)$ lifts to $\tilde{c}: \U(n+1) \sma \U(n+1) \rightarrow \Omega\V_{n+1}$ along the map $j: \Omega\V_{n+1} \rightarrow \U(n+1)$ in the fiber sequence 
\[
\Omega\U \stackrel{\Omega\pi}{\longrightarrow} \Omega\V_{n+1} \stackrel{j}{\longrightarrow} \U(n+1) \longrightarrow \U\stackrel{\pi}{\longrightarrow} \V_{n+1}.
\]

We map $\Sigma^{2n}\CP^2$ to the diagram \eqref{eq:commdiagn}.
Let $N_k$ be the subgroup of $[\Sigma^{2n}\CP^2, \U(n+1)]$ generated by
\[
i\circ \langle k\epsilon, \bar{x} \rangle = i \circ c \circ (k\epsilon \sma \bar{x}) \quad \text{and} \quad i\circ \langle k\epsilon, \bar{x}' \rangle = i \circ c \circ (k\epsilon \sma \bar{x}').
\]
The size of $\Im (\partial_k)_*$ is the same as that of $N_k$. 
Because the square in \eqref{eq:commdiagn} commutes, we have
\[
i\circ \langle k\epsilon, \bar{x} \rangle = c \circ ((i\circ k\epsilon) \sma (i\circ \bar{x})) \quad \text{and} \quad i\circ \langle k\epsilon, \bar{x}' \rangle = c \circ ((i\circ k\epsilon) \sma (i\circ \bar{x}')).
\]
Let $M_k$ be the subgroup of $[\Sigma^{2n}\CP^2, \Omega\V_{n+1}]$ generated by
\[
\tilde{c} \circ ((i\circ k\epsilon) \sma (i\circ \bar{x})) \quad \text{and} \quad \tilde{c} \circ ((i\circ k\epsilon) \sma (i\circ \bar{x}')).
\]
We have a diagram of abelian groups:
\begin{equation}
\begin{tikzcd}[column sep = small]
{[\Sigma^{2n}\CP^2, \Omega\U]} \ar[r, "(\Omega\pi)_*"'] & {[\Sigma^{2n}\CP^2, \Omega\V_{n+1}]} \ar[r, "j_*"'] & {[\Sigma^{2n}\CP^2, \U(n+1)]} \ar[r] & {[\Sigma^{2n}\CP^2, \U]} = 0 \\
 & M_k \ar[u, hook] & N_k \ar[u, hook] & 
\end{tikzcd}
\end{equation}
where the top row is exact and the vertical maps are subgroup inclusions. It follows that
\[
N_k \cong M_k/\left( M_k\cap \Ker j_* \right) \cong M_k/\left( M_k\cap \Im (\Omega\pi)_* \right).
\]
Therefore, to find the size of $\Im (\partial_k)_*$, which is the size of $N_k$, it suffices to find the index of $M_k\cap \Im (\Omega\pi)_*$ as a subgroup of $M_k$.

\smallskip

Let us first work out the case $k=1$. Recall that elements of $[\Sigma^{2n}\CP^2, \Omega\V_{n+1}]$ correspond exactly to pairs of integers $(a, b)$ with $a$ even, by \Cref{lambdainjnodd}.

\begin{prop} \label[prop]{M1basisnodd}
The subgroup $M_1 \subset [\Sigma^{2n}\CP^2, \Omega\V_{n+1}]$ is generated by $((n-1)!, (n-1)!)$ and $(0, 2\cdot (n-1)!)$.
\end{prop}

\begin{proof}
The job is to calculate 
\[
\lambda \circ \tilde{c} \circ ((i\circ \epsilon) \sma (i\circ \bar{x})) \quad \text{and} \quad \lambda \circ \tilde{c} \circ ((i\circ \epsilon) \sma (i\circ \bar{x}')).
\]
We use the Hamanaka--Kono result \cite[Proposition 5.2]{HK03} that
\[
\tilde{c}^*(a_{2n}) = \sum_{j+k = n-1} x_{2j+1} \otimes x_{2k+1}.
\]
In our case, $\lambda = (a_{2n+2}\times a_{2n+4})_*$, so
\[
\lambda \circ \tilde{c} \circ ((i\circ \epsilon) \sma (i\circ \bar{x})) = (\epsilon^*x_{2n-1}\otimes \bar{x}^*x_3, \, \epsilon^*x_{2n-1}\otimes \bar{x}^*x_5).
\]
Since $\epsilon^*x_{2n-1} = (n-1)! \iota$ (where $\iota\in \Hmr^{2n-1}(\Smr^{2n-1}; \ZZ)$ is the generator), $\bar{x}^*x_3 = \sigma t$ and $\bar{x}^*x_5 = \sigma t^2$ (where $\sigma$ denotes the suspension isomorphism), we find that 
\[
\lambda \circ \tilde{c} \circ ((i\circ \epsilon) \sma (i\circ \bar{x})) = ((n-1)!, (n-1)!).
\]
Similarly, since $(\bar{x}')^*x_3 = 0$ and $(\bar{x}')^*x_5 = 2 \sigma t^2$, we find that
\[
\lambda \circ \tilde{c} \circ ((i\circ \epsilon) \sma (i\circ \bar{x})) = (0, 2\cdot (n-1)!).
\]
The proof is now complete.
\end{proof}

Write 
$\alpha = ((n-1)!, (n-1)!)$ and $\beta = (0, 2\cdot (n-1)!)$. 
\Cref{M1basisnodd} says that $\{\alpha, \beta\}$ is a basis for $M_1$.
Recall that $u=(2, n+2)$ and $v=(0, 1)$, which form a basis for $[\Sigma^{2n}\CP^2, \Omega\V_{n+1}]$. Note that
\[
4\alpha + n\beta = 2 \cdot (n-1)! \cdot u,
\]
and that
\[
n\alpha + \frac{n^2-1}{4}\beta = \frac{n!}{2}u - \frac{(n-1)!}{2}v = \frac{(n-1)!}{2} \cdot (nu - v).
\]
So if we let $u':=u$, and $v':=nu-v$, then $\{u', v'\}$ is also basis for the  $[\Sigma^{2n}\CP^2, \Omega\V_{n+1}]$, as the change-of-basis matrix has determinant $-1$. Also note that
\[
\det \begin{bmatrix}
4 &  n \\
n & \frac{n^2-1}{4}
\end{bmatrix} = -1.
\]
So $\{2\cdot (n-1)! \cdot u', \, \frac{(n-1)!}{2}\cdot v' \}$ is a basis for $M_1$. 
Finally, note that
\[
\left(\frac{(n+1)!}{2}u, (n+2)!v\right) \begin{bmatrix}
1 & 2n(n+2) \\
0 & -1
\end{bmatrix} = \left(\frac{(n+1)!}{2}u', (n+2)!v'\right)
\]
where the change-of-basis matrix above has determinant $-1$. So $\{\frac{(n+1)!}{2}u', \, (n+2)!v'\}$ forms a basis for $\Im (\Omega\pi)_*$ by \Cref{imomegapistarnodd}.
To sum up, we have proved:

\begin{prop} \label[prop]{basesnodd}
A $\ZZ$-basis for the rank-two free abelian group $[\Sigma^{2n}\CP^2, \Omega\V_{2n+1}]$ and its related subgroups, can be presented as follows:

\smallskip

\begin{center}
\begin{tabular}{c||c}
\hline
Group & Basis \\
\hline
$[\Sigma^{2n}\CP^2, \Omega\V_{n+1}]$ & $\{u', v'\}$ \\
\hline
$\Im (\Omega\pi)_*$ & $\{\frac{(n+1)!}{2}u', (n+2)!v'\}$ \\
\hline
$M_1$ & $\{2\cdot (n-1)! u', \, \frac{(n-1)!}{2} v'\}$ \\
\hline
$M_k$ & $\{2\cdot (n-1)! ku', \, \frac{(n-1)!}{2} kv'\}$ \\
\hline
\end{tabular}
\end{center}
\end{prop}

We are now ready to prove our first main result, \Cref{thm:lowerboundnodd}.

\begin{proof}[Proof of \Cref{thm:lowerboundnodd}]
By \Cref{basesnodd}, the index of $M_k\cap \Im (\Omega\pi)_*$ in $M_k$ equals
\begin{align*}
&  \frac{\lcm\left((n+1)!/2, \, 2\cdot (n-1)!k\right)}{2(n-1)!k} \cdot \frac{\lcm\left((n+2)!, ((n-1)!/2)k \right)}{((n-1)!/2)k} \\[2mm]
 = & \frac{(n+1)!/2}{\gcd\left( (n+1)!/2, \, 2 (n-1)!k \right)} \cdot \frac{(n+2)!}{ \gcd \left( (n+2)!, ((n-1)!/2)k \right)} \\[2mm]
 = & \frac{n(n+1)/2}{\gcd\left(n(n+1)/2, 2k\right)} \cdot \frac{2n(n+1)(n+2)}{\gcd\left(2n(n+1)(n+2), k \right)}.
\end{align*}
Note that
\[
\gcd\left(n(n+1)/2, 2k\right) = \begin{cases}
\gcd\left(n(n+1)/2, k\right) & \text{if } n\equiv 1 \text{ mod } 4 \\
2 \gcd\left(n(n+1)/4, k\right) & \text{if } n\equiv 3 \text{ mod } 4.
\end{cases}
\]
Since both $n(n+1)/2$ and $n(n+1)/4$ can divide $2n(n+1)(n+2)$, we conclude that if $\Gcal_k\simeq \Gcal_{\ell}$ then $\gcd(2n(n+1)(n+2), k) = \gcd(2n(n+1)(n+2), \ell)$.
\end{proof}

\smallskip

\subsection{The case when \texorpdfstring{$n$}{n} is even} \label{subsec:evenn}
We now deal with the case when $n$ is even.
In this case, the $(2n+4)$-skeleton of $\Omega\V_{n+1}$ is no longer a wedge of two spheres but instead (a shift of) the cone of $\eta$, since its bottom cohomology supports a non-trivial action of $\Sq^2$.
This results in a different description of the image of $\lambda$ in the diagram \eqref{eq:testUnplus1}, already established by Hamanaka--Kono \cite[Lemma~2.1]{HK06}:
\begin{prop} \label[prop]{lambdainjneven}
The homomorphism $\lambda$ in the diagram \eqref{eq:testUnplus1} is injective, whose image is
\[
\{(a, b)\in \ZZ\oplus \ZZ: a \equiv b \text{ mod } 2\}.
\]
\end{prop}

Following \Cref{lambdainjneven}, we shall write an element of $[\Sigma^{2n}\CP^2, \Omega\V_{n+1}]$ as a pair of integers $(a, b)$ where $a, b$ have the same parity. In particular, $u:=(1, 1)$ and $v:=(0, 2)$ form a $\ZZ$-basis for $\Im \lambda$.
Note that \Cref{lambdainjneven} does not affect the calculations in \Cref{imomegapistarnodd} and \Cref{M1basisnodd}, but only changes the linear algebra problem we need to solve.
So we still have that the image of the homomorphism
\[
(\Omega\pi)_*: \KUt^{0}(\Sigma^{2n}\CP^2) \longrightarrow [\Sigma^{2n}\CP^2, \Omega\V_{n+1}]
\] 
is generated by 
\[
\left( (n+1)!, \, (n+2)!/2 \right) \quad \text{and} \quad (0, \, (n+2)!)
\]
which follows from the proof of \Cref{imomegapistarnodd}, and that
the subgroup $M_1 \subset [\Sigma^{2n}\CP^2, \Omega\V_{n+1}]$ is generated by 
\[
\alpha:= ((n-1)!, (n-1)!) \quad \text{and} \quad \beta:= (0, 2\cdot (n-1)!)
\]
which is the content of \Cref{M1basisnodd}. We now solve the linear algebra problem we need, to prove our main result.

\smallskip

When $n$ equals $0$ mod $4$, the number $(n+2)/2$ is odd, and therefore 
\[
u':=(1, (n+2)/2) = u + \frac{n}{4}v \quad \text{and} \quad v':=(0, 2) = v
\]
form a basis for $[\Sigma^{2n}\CP^2, \Omega\V_{n+1}]$. Since
\[
\left( (n+1)!, \, (n+2)!/2 \right) = (n+1)! u' \quad \text{and} \quad (0, \, (n+2)!) = \frac{(n+2)!}{2} v',
\]
we see that $\{(n+1)! u', \, ((n+2)!/2) v'\}$ is a basis for $\Im (\Omega\pi)_*$.
Note that 
\[
\alpha + \frac{n}{4}\beta \quad \text{and} \quad \beta
\]
still generate $M_1$, and that
\[
\alpha + \frac{n}{4}\beta = (n-1)!u' \quad \text{and} \quad \beta = (n-1)!v'.
\]
To sum up, for $n$ equals $0$ mod $4$ we have obtained:
\begin{equation} \label{eq:summary0mod4}
\begin{tabular}{c||c}
\hline
Group & Basis \\
\hline
$[\Sigma^{2n}\CP^2, \Omega\V_{n+1}]$ & $\{u', v'\}$ \\
\hline
$\Im (\Omega\pi)_*$ & $\{(n+1)!u', ((n+2)!/2)v'\}$ \\
\hline
$M_1$ & $\{(n-1)! u', \, (n-1)! v'\}$ \\
\hline
$M_k$ & $\{(n-1)! ku', \, (n-1)! kv'\}$ \\
\hline
\end{tabular}
\end{equation}

\smallskip

When $n$ equals $2$ mod $4$, the number $(n+2)/2$ is even. This time we let
\[
u'':= (1, \frac{n}{2}) = u + \frac{n-2}{4}v \quad \text{and} \quad v'':= (2, n+2) = 2u+\frac{n}{2}v.
\]
Since the change-of-basis matrix from $\{u, v\}$ to $\{u'', v''\}$ has determinant $1$, the set $\{u'', v''\}$ forms a basis for the group $[\Sigma^{2n}\CP^2, \Omega\V_{n+1}]$.
Furthermore, note that $\Im (\Omega\pi)_*$ can alternatively be generated by
\[
\left( (n+1)!, \frac{(n+2)!}{2} \right) = \frac{(n+1)!}{2} v''
\]
and
\[
(n+2)\cdot \left( (n+1)!, \frac{(n+2)!}{2} \right) - (0, (n+2)!) = (n+2)! u''
\]
because the change-of-basis matrix
\[
\begin{bmatrix}
 1 & 0 \\
 (n+2) & -1
\end{bmatrix}
\]
has determinant $-1$. Thus $\{(n+2)! u'', \frac{(n+1)!}{2} v''\}$ is a basis for $\Im (\Omega\pi)_*$.
Finally, note that instead of using $\{\alpha, \beta\}$ as a basis for $M_1$, here one can use
\[
\alpha + \frac{n-2}{4} \beta = (n-1)! u''
\]
and
\[
2\alpha + \frac{n}{2}\beta = (n-1)! v'',
\]
again because the change-of-basis matrix has determinant $1$. To sum up, for $n$ equals $2$ mod $4$, we have:
\begin{equation} \label{eq:summary2mod4}
\begin{tabular}{c||c}
\hline
Group & Basis \\
\hline
$[\Sigma^{2n}\CP^2, \Omega\V_{n+1}]$ & $\{u'', v''\}$ \\
\hline
$\Im (\Omega\pi)_*$ & $\{(n+2)!u'', ((n+1)!/2)v''\}$ \\
\hline
$M_1$ & $\{(n-1)! u'', \, (n-1)! v''\}$ \\
\hline
$M_k$ & $\{(n-1)! ku'', \, (n-1)! kv''\}$ \\
\hline
\end{tabular}
\end{equation}

We are now ready to prove our second main result, \Cref{thm:lowerboundneven}.

\begin{proof}[Proof of \Cref{thm:lowerboundneven}]
Suppose that $n$ is $0$ mod $4$. By the table \eqref{eq:summary0mod4}, the index of $M_k \cap \Im (\Omega\pi)_*$ in $M_k$ equals
\begin{align*}
& \frac{\lcm\left((n+1)!, \, (n-1)!k\right)}{(n-1)!k} \cdot \frac{\lcm\left((n+2)!/2, (n-1)!k \right)}{(n-1)!k} \\[2mm]
= & \frac{(n+1)!}{\gcd\left((n+1)!, \, (n-1)!k\right)} \cdot \frac{(n+2)!/2}{\gcd\left((n+2)!/2, (n-1)!k \right)} \\[2mm]
= & \frac{n(n+1)}{\gcd\left(n(n+1), k\right)} \cdot \frac{n(n+1)(n+2)/2}{\gcd\left(n(n+1)(n+2)/2, k \right)}.
\end{align*}
Note that $n(n+1)$ divides $n(n+1)(n+2)/2$. So if $\Gcal_k \simeq \Gcal_{\ell}$ then $\gcd\left(n(n+1)(n+2)/2, k \right)= \gcd\left(n(n+1)(n+2)/2, \ell \right)$.

\smallskip

Suppose that $n$ is $2$ mod $4$. By the table \eqref{eq:summary2mod4}, the index of $M_k \cap \Im (\Omega\pi)_*$ in $M_k$ equals
\begin{align*}
 & \frac{\lcm\left((n+2)!, \, (n-1)!k\right)}{(n-1)!k} \cdot \frac{\lcm\left((n+1)!/2, (n-1)!k \right)}{(n-1)!k} \\[2mm]
 = & \frac{(n+2)!}{\gcd\left((n+2)!, \, (n-1)!k\right)} \cdot \frac{(n+1)!/2}{\gcd\left((n+1)!/2, (n-1)!k \right)} \\[2mm]
 = & \frac{n(n+1)(n+2)}{\gcd\left(n(n+1)(n+2), k\right)} \cdot \frac{n(n+1)/2}{\gcd\left(n(n+1)/2, k \right)}.
\end{align*}
Note that $n(n+1)/2$ divides $n(n+1)(n+2)$. So if $\Gcal_k \simeq \Gcal_{\ell}$ then $\gcd\left(n(n+1)(n+2), k \right)= \gcd\left(n(n+1)(n+2), \ell \right)$. 
\end{proof}


\section{Homotopy types of \texorpdfstring{$\SU(n)$}{SUn}-gauge groups over \texorpdfstring{$\Smr^{2m+1}$}{S2mplus1}, \texorpdfstring{$m\geq n$}{mgreatern}} \label{sec:unstable}


This section is dedicated to the unstable case. In \Cref{subsec:Samelson} we compute a Samelson product localized at an odd prime $p$, generalizing the result of Hamanaka \cite[Theorem 9]{Ha07}. In \Cref{subsec:application} we apply our result to detect the homotopy types of $\SU(n)$-gauge groups over $\Smr^{2m+1}$, in the unstable range $m\geq n$.

\subsection{A key Samelson product} \label{subsec:Samelson}

The following notations and conventions are to be applied throughout this subsection:

\begin{itemize}
\item[-] All calculations are localized at an odd prime $p$, unless otherwise stated.
\item[-] The generator of $\pi_{2i+1}\SU(n) \cong \ZZ$, $1\leq i \leq n-1$, is denoted by $\epsilon_i$.
\item[-] The generator of $\pi_{2n+2j}\SU(n) \cong \ZZ/p^{\nu_p((n+j)!)}$, $0\leq j\leq p-2$, is written as $\omega_{n+j}$.
\item[-] The Stiefel manifold $\SU/\SU(n)$ is denoted by $\V_n$.
\item[-] Write $a_{2n+2j}$ for the standard cohomology class generating $\Hmr^{2n+2j}(\Omega\V_n; \ZZ)$, which transgresses to the class $x_{2n+2j+1}\in \Hmr^{2n+2j+1}(\V_n;\ZZ)$.
\item[-] Write $\alpha_1$ for the generator of the stable stem $\pi_{2p-3}^s\cong \ZZ/p$ (and recall that $p$-locally $\pi_k^s = 0$ for $0<k<2p-3$, and for $2p-3<k<4p-5$).
\item[-] Write $N$ for the sum $i+j$. 
\end{itemize}

\begin{thm} \label{thm:Sam}
Let $N\in \{p-2, p-1, \cdots, 2p-4\}$. Suppose that $n \equiv p-2-N$ mod $p$, and that $p\leq n+1$. Then the Samelson product $\langle \epsilon_i, \omega_{n+j} \rangle \in \pi_{2n+2N+1} \SU(n)$ equals 
\[
i! \, (n-1-(p-2-j)) \, \omega_{n+N-(p-2)} \circ \alpha_1.
\]
\end{thm}

We first establish the lemmas needed to prove the above main result.

\begin{lemma} \label[lemma]{factor1}
When $n \equiv p-2-N$ mod $p$ and $p\leq n+1$, the homotopy group $\pi_{2n+2N+1}\SU(n)$ is isomorphic to $\ZZ/p$, with $\omega_{n+N-(p-2)} \circ \alpha_1$ as its generator.
\end{lemma}

\begin{proof}
Note that $\pi_{2n+2N+1}\SU(n)$ is the set $\Vect_n^0(\Smr^{2n+2N+2})$ of rank $n$ stably trivial complex vector bundles over $\Smr^{2n+2N+2}$. Indeed, since 
\[
\pi_{2n+2N+1}\SU(n) \cong \pi_{2n+2N+2}\BSU(n) = [\Smr^{2n+2N+2}, \BSU(n)] \cong [\Smr^{2n+2N+2}, \BU(n)],
\]
this is the set $\Vect_n(\Smr^{2n+2N+2})$ of rank $n$ complex vector bundles over $\Smr^{2n+2N+2}$. Any such bundle is necessarily stably trivial, since the bundle must have vanishing Chern classes for degree reason, and since the Chern character is injective for even spheres.
By \cite[Theorem 2.1]{Hu23}, when $2N+2\leq 4n$ we have
\[
\Vect_n^0(\Smr^{2n+2N+2}) \cong \pi_{2n+2N+2}^s(\Sigma\CP_n^{\infty}) \cong \pi_{2n+2N+1}^s(\CP_n^{\infty}).
\]
Note that $2N+2\leq 4n$ is guaranteed by the condition $p\leq n+1$.
Also observe that
\[
\pi_{2n+2N+1}^s(\CP_n^{\infty}) \cong \pi_{2n+2N+1}^s(\CP_n^{n+N+1}).
\]
This follows from the cofiber sequence of spectra:
\[
\Sigma^{-1}\CP_{n+N+2}^{\infty} \longrightarrow \CP_n^{n+N+1} \longrightarrow \CP_n^{\infty} \longrightarrow \CP_{n+N+2}^{\infty}
\]
where $\Sigma^{-1}\CP_{n+N+2}^{\infty}$ is $(2n+2N+2)$-connected.

Furthermore, note that $\CP_n^{n+N+1}$ is a wedge of spheres and suspensions of $\Cmr(\alpha_1)$.
Indeed, the top cell in the stunted projective spectrum is attached via a map 
\[
\Smr^{2n+2N+1} \longrightarrow \Smr^{2n}.
\]
The dimension difference between the source and target spheres equals $2N+1$, which is no greater than $4p-7$ since the maximal value of $N$ is $2p-4$. So no attaching map in the structure of $\CP_n^{n+N+1}$ can involve $\alpha_2\in \pi_{4p-5}^s$ or elements in higher stems.
In particular, we claim that $\Smr^{2n+2N-(2p-4)}$ is a summand of $\CP_n^{n+N+1}$. Indeed, for dimension reasons this cell cannot be attached to any lower one via an $\alpha_1$-attaching map. Also, by our assumption $n+N-(p-2)$ is zero modulo $p$, so the corresponding $\Hmr\FF_p$-cohomology class does not support a nonzero \mbox{$\Sp^1$-action}. Therefore no higher-dimensional cell can be attached to it via an $\alpha_1$-attaching map.

Finally, note that $\Smr^{2n+2N-(2p-4)}$ is the only summand of $\CP_n^{n+N+1}$ to which $\Smr^{2n+2N+1}$ can map non-trivially, just for dimension reasons.
To sum up, we have proved that
\[
\pi_{2n+2N+1}\SU(n) \cong \pi_{2n+2N+1}^s(\CP_n^{n+N+1}) \cong \pi_{2n+2N+1}(\Smr^{2n+2N-(2p-4)}) \cong \ZZ/p \cdot \alpha_1.
\]
So the generator of $\pi_{2n+2N+1}\SU(n)$ factors through $\alpha_1$, and the desired result follows.
\end{proof}

\smallskip

We need the following lemma, originally proved by Hamanaka \cite{Ha07}.

\begin{lemma}[{\cite[Lemma 8]{Ha07}}] \label[lemma]{lift}
For $0\leq j\leq 2p-2$, $\omega_{2n+2j}: \Smr^{2n+2j} \rightarrow \SU(n)$ lifts against the standard inclusion
\[
\Sigma \CP^{n-1-(p-2-j)} \longrightarrow \Sigma \CP^{n-1} \longrightarrow \SU(n).
\]
\end{lemma}



Denote by $\nu_{2n+2j}: \Smr^{2n+2j} \rightarrow \Sigma \CP^{n-1-(p-2-j)}$ the lift of $\omega_{2n+2j}$. This makes the bottom left triangle in the following diagram commute.
\begin{equation} \label{eq:diagram}
\begin{tikzcd}
& \phantom{} & \\
 & \Smr^{2i+1} \sma \Smr^{2n-1+2j-(2p-4)}  \ar[rdd, "f", bend left = 35]  & \\
 & \Smr^{2i+1} \sma \Sigma\CP_{n-1-(N-j)}^{n-1-(p-2-j)} \ar[u, "\id \sma r"] \ar[rd, "g", bend left = 10] & \\  
 & \Smr^{2i+1} \sma \Sigma\CP^{n-1-(p-2-j)} \ar[d, "\epsilon_i \sma l"'] \ar[r, "h"] \ar[u, "\id \sma q"] & \Omega\V_n \ar[d]  \arrow[uuul, phantom, "\not\circlearrowright"'{color=red}]  \\  
   \Smr^{2n+2N+1} = \Smr^{2i+1}  \sma \Smr^{2n+2j} \ar[r, "\epsilon_i \sma \omega_{2n+2j}"'] \ar[ru, "\id\sma \nu_{2n+2j}"', bend left = 10] \ar[ruu, "\mu", bend left = 25] \ar[ruuu, "\lambda", bend left = 40] & \SU(n)\sma \SU(n) \ar[ru, "\tilde{c}"'] \ar[r, "c"'] & \SU(n) \\  
\end{tikzcd}
\end{equation}
Some other notations involved in the above diagram are explained below:
\begin{itemize}
 \item[-] $\mu$ is the composite $(\id\sma q) \circ (\id\sma \nu_{2n+2j})$, and $\lambda$ is the composite $(\id\circ r) \circ \mu$.
 \item[-] $l$ is the inclusion, and $h$ is the composite $\tilde{c} \circ (\epsilon_i \sma l)$. 
 \item[-] $q$ is the quotient of $\Sigma\CP^{n-1-(p-2-j)}$ by the skeleton $\Sigma\CP^{n-2-(N-j)}$, and $r$ is the quotient of $\Sigma\CP_{n-1-(N-j)}^{n-1-(p-2-j)}$  onto its top cell.
\end{itemize}

\begin{lemma} \label[lemma]{factor2}
There exists a map 
\[
g: \Smr^{2i+1} \sma \Sigma\CP_{n-1-(N-j)}^{n-1-(p-2-j)} \longrightarrow \Omega\V_n
\] 
so that $h = g \circ (\id\sma q)$.
\end{lemma}

\begin{proof}
The map $\id\sma q$ is part of the cofiber sequence:
\[
\Smr^{2i+1}\sma \Sigma\CP^{n-2-(N-j)} \longrightarrow \Smr^{2i+1} \sma \Sigma\CP^{n-1-(p-2-j)} \stackrel{\id\sma q}{\longrightarrow} \Smr^{2i+1} \sma \Sigma\CP_{n-1-(N-j)}^{n-1-(p-2-j)}.
\]
Note that:
\begin{itemize}
\item[-] $\Omega\V_n$ is $(2n-1)$-connected, and that
\item[-] $\Smr^{2i+1}\sma \Sigma\CP^{n-2-(N-j)}$ has dimension $2i+1+2n-4-2(N-j)+1 = 2n-2$.
\end{itemize}
Since $2n-2 < 2n-1$, the first map in the cofiber sequence becomes null after composing with $h$. So $h$ factors through $\id \sma q$.
\end{proof}

Furthermore, we note that $\Sigma\CP_{n-1-(N-j)}^{n-1-(p-2-j)}$ is a wedge of spheres. Indeed, its top cell has dimension $2(n-1-(p-2-j))+1$ and its bottom cell has dimension $2(n-1-(N-j))+1$. The dimension difference is $2(N-(p-2))$, which is at most $2((2p-4)-(p-2)) = 2p-4$, and this is strictly smaller than $|\alpha_1| = 2p-3$. So the quotient map $r$ in diagram \eqref{eq:diagram} admits a splitting
\[
s: \Smr^{2n+2N-(2p-4)} = \Smr^{2i+1} \sma \Smr^{2n-1+2j-(2p-4)} \longrightarrow \Smr^{2i+1} \sma \Sigma\CP_{n-1-(N-j)}^{n-1-(p-2-j)},
\]
namely the inclusion of the wedge factor.
The map $f$ in diagram \eqref{eq:diagram} is defined as the composite $g\circ s$.
Note that $f$ does not make the top-right triangle in diagram \eqref{eq:diagram} commute, but rather satisfies $f\circ \lambda = g\circ \mu$. 

\smallskip

To prove \Cref{thm:Sam}, we need to further identify the map  $\lambda \in \pi_{2n+2N+1}(\Smr^{2n+2N-(2p-4)}) =\pi_{2p-3}^s$, and the map $f: \pi_{2n+2N-(2p-4)}(\Omega\V_n)$ in diagram \eqref{eq:diagram}. The next two propositions complete this task.

\begin{prop} \label[prop]{lambda}
The map $\lambda$ equals $(n-1-(p-2-j))\alpha_1$.
\end{prop}

\begin{proof}
Since $\lambda \in \pi_{2p-3}^s \cong \ZZ/p \cdot \alpha_1$, we know $\lambda = m\alpha_1$ for some $m\in \ZZ/p$. By \cite[Lemma~8]{Ha07}, the attaching map of the top cell of $\Sigma \CP^{n+j}$ factors through $\nu_{n+j}$, so the coefficient $m$ can be detected by the action of the Steenrod power $\Sp^1$ on the cohomology class $\sigma x^{n-1-(p-2-j)}\in \Hmr^{2n-1-2(p-2-j)}(\Sigma \CP^{n+j})$. 
Since $\Sp^1(\sigma x^{n-1-(p-2-j)}) = (n-1-(p-2-j)) \sigma x^{n+j}$, we conclude that $m = (n-1-(p-2-j))$.
\end{proof}

\begin{prop} \label[prop]{f}
The map $f$ equals $i! \, \tilde{\omega}_{n+N-(p-2)}$, where 
\[
\tilde{\omega}_{n+N-(p-2)} \in \pi_{2n+2N-(2p-4)}(\Omega\V_n)
\]
 is the lift of $\omega_{n+N-(p-2)}\in \pi_{2n+2N-(2p-4)}(\SU(n))$ so that it pulls back the cohomology class 
 \[
 a_{2n+2N-(2p-4)}\in \Hmr^{2n+2N-(2p-4)}(\Omega\V_n)
 \]
  to the standard class 
  \[
  \iota_{2n+2N-(2p-4)}\in \Hmr^{2n+2N-(2p-4)}(\Smr^{2n+2N-(2p-4)}).
  \]
\end{prop}

\begin{proof}
We compute $f^*a_{2n+2N-(2p-4)}$. Note that $\id\sma r$ in diagram \eqref{eq:diagram} induces a cohomology isomorphism in degree $2n+2N-(2p-4)$, sending the standard class $\iota_{2n+2N-(2p-4)}$ to the class $\iota_{2i+1}\otimes \sigma x^{2n-2+2j-(2p-4)}$. The inverse of $(\id\sma r)^*$ is precisely $(\id \sma s)^*$. Since $f$ is defined as $g\circ (\id\sma s)$, $f^* = (\id\sma s)^*\circ g^*$. Because $(\id\sma s)^*$ is an isomorphism, it suffices to compute $g^*a_{2n+2N-(2p-4)}$. Similarly, because $\id\sma q$ induces a cohomology isomorphism in degree $2n+2N-(2p-4)$, it suffices to compute $h^*a_{2n+2N-(2p-4)}$. We find that
\begin{align*}
h^*(a_{2n+2N-(2p-4)}) &= ((\epsilon_i \sma l)^* \circ \tilde{c}^*)(a_{2n+2N-(2p-4)}) \\
 &= (\epsilon_i \sma l)^*\left(\sum_{u+v = n+N-(p-1)}x_{2u+1}\otimes x_{2v+1}\right) \\
 &= \epsilon_i^*(x_{2i+1}) \otimes \sigma x_{2n-2+2j-(2p-4)} \\
 &= i! \, \iota_{2i+1} \otimes \sigma x_{2n-2+2j-(2p-4)},
\end{align*}
where the second equality follows from \cite[Proposition~4]{Ha07}.
It follows that 
\[
f^*a_{2n+2N-(2p-4)} = i! \, \iota_{2n+2N-(2p-4)},
\]
and hence that $f = i! \, \tilde{\omega}_{n+N-(p-2)}$.
\end{proof}

We are now ready to prove the main result of this section.

\begin{proof}[Proof of \Cref{thm:Sam}]
By \Cref{factor1}, \Cref{lift} and \Cref{factor2}, all triangles in diagram \eqref{eq:diagram} commute except the one at the top right. However, recall that we do have $f\circ \lambda = g\circ \mu$, so the commutativity of the bottom part of diagram \eqref{eq:diagram} implies that $\left\langle \epsilon_i, \omega_{2n+2j} \right\rangle = c\circ (\epsilon_i \sma \omega_{2n+2j})$ is exactly the composition of $f\circ \lambda$ with the standard map $\Omega\V_n\rightarrow \SU(n)$. It then follows from \Cref{lambda} and \Cref{f} that $\left\langle \epsilon_i, \omega_{2n+2j} \right\rangle = i! \, (n-1-(p-2-j)) \, \omega_{n+N-(p-2)} \circ \alpha_1$.
\end{proof}



\subsection{Homotopy types of unstable \texorpdfstring{$\SU(n)$}{SUn}-gauge groups over odd spheres} \label{subsec:application}

We now apply \Cref{thm:Sam} to obtain lower bounds for the number of homotopy types of $\SU(n)$-gauge groups over odd spheres $\Smr^{2m+1}$, in the unstable range $m\geq n$. 

\smallskip

Write $m=n+j$, $j\geq 0$. The principal $\SU(n)$-bundles over $\Smr^{2n+2j+1}$ are classified as
\[
\Prin_{\SU(n)}(\Smr^{2n+2j+1}) \cong [\Smr^{2n+2j+1}, \BSU(n)] \cong \pi_{2n+2j}\SU(n).
\]
For $\sfP \in \Prin_{\SU(n)}(\Smr^{2n+2j+1})$, write $f_{\sfP} \in \pi_{2n+2j}\SU(n)$ for the corresponding classifying map. Denote by $\Gcal_{\sfP}$ for the gauge group of $\sfP$.
Recall that we have the fiber sequence:
\begin{equation} \label{eq:funfib}
\begin{tikzcd}
\Gcal_{\sfP} \ar[r] & \SU(n) \ar[r, "\partial_{\sfP}"'] & \Omega_0^{2n+2j}\SU(n) \ar[r] & \B\Gcal_{\sfP} \ar[r] & \BSU(n)
\end{tikzcd}
\end{equation}
where the connecting map $\partial_{\sfP}$ is the adjoint of the Samelson product
\[
\begin{tikzcd}
\left\langle f_{\sfP}, 1_{\SU(n)}  \right\rangle: \Smr^{2n+2j} \sma \SU(n) \ar[rr, "f_{\sfP}\sma 1_{\SU(n)}"'] & &  \SU(n)\sma\SU(n) \ar[r, "c"'] & \SU(n).
\end{tikzcd}
\]

\begin{thm} \label{thm:appl}
Let $p$ be an odd prime so that $j \leq p-2$. Suppose that $p\leq n+1$, and that the value of $n$ mod $p$ belongs to the set $\{0, p-1, \cdots, p-j\}$. There are at least two different $p$-local homotopy types of $\Gcal_{\sfP}$, as $\sfP$ ranges over principal $\SU(n)$-bundles over $\Smr^{2n+2j+1}$.
\end{thm}

\begin{proof}
Since the value of $n$ mod $p$ belongs to the set $\{0, p-1, \cdots, p-j\}$, there is a unique $N\in \{p-2, p-1, \cdots, p-2+j\}$ so that $n \equiv p-2-N$ mod $p$. Let $i:= N-j$. Note that $i$ satisfies $0\leq i \leq p-2$. Since $p\leq n+1$ is assumed, we also have $i\leq p-2\leq n-1$. We use $\Smr^{2i+1}$ as the test complex. Mapping $\Smr^{2i+1}$ to the fiber sequence \eqref{eq:funfib}, we obtain an exact sequence
\[
\begin{tikzcd}
\pi_{2i+1}\SU(n) \ar[r, "(\partial_{\sfP})_*"'] & \pi_{2i+1}\Omega^{2n+2j}\SU(n) = \pi_{2n+2N+1}\SU(n) \ar[r] & \pi_{2i+1}\B\Gcal_{\sf{P}} \ar[r] & 0
\end{tikzcd}
\]
where the connecting homomorphism $(\partial_{\sfP})_*$ is such that
\[
\left( \alpha\in \pi_{2i+1}\SU(n) \right) \longmapsto \left( \left\langle \alpha, f_{\sfP} \right\rangle \in \pi_{2n+2N+1}\SU(n)  \right).
\]
Note that the last group in the exact sequence is zero, because $\pi_{2i+1}\BSU(n) = 0$ for $i\leq n-1$.
Write $\Gcal_0$ for the gauge group of the trivial bundle, and $\Gcal_1$ for the gauge group of the bundle classified by (any integral lift of) $\omega_{n+j}$. Write $\partial_0$ and $\partial_1$ for the connecting maps, respectively. To prove our results it suffices to show that $\Gcal_0 \not\simeq \Gcal_1$.
If $\Gcal_0 \simeq \Gcal_1$ then $\pi_{2i+1}\B\Gcal_{0} \cong \pi_{2i+1}\B\Gcal_{1}$, and $\Im (\partial_0)_*$ and $\Im (\partial_1)_*$ must therefore be of the same size. 
Since $(\partial_0)_*$ is the zero homomorphism, it suffices to show that there is some $\alpha\in \pi_{2i+1}\SU(n)$ so that $(\partial_1)_*(\alpha) \neq 0$. We claim that the generator $\epsilon_i \in \pi_{2i+1}\SU(n)$ can serve as such an $\alpha$. Indeed, $p$-locally we have
\[
(\partial_1)_*(\epsilon_i) = \left\langle \epsilon_i, \omega_{n+j} \right\rangle,
\]
which by \Cref{thm:Sam} equals
\[
i! \, (n-1-(p-2-j)) \, \omega_{n+N-(p-2)} \circ \alpha_1.
\]
Recall from \Cref{factor1} that $\pi_{2n+2N+1}\SU(n)$ is $p$-locally generated by $\omega_{n+N-(p-2)} \circ \alpha_1$, and note that
\[
i! \, (n-1-(p-2-j)) \equiv i! \, ((p-2-N)-1-(p-2-j)) \equiv i!\, (j-N-1) \equiv i! \, (-i-1) \text{ mod } p.
\]
Recall that $i \leq p-2$, so the value $i! \, (-i-1)$ is not divisible by $p$. It follows that $(\partial_1)_*(\epsilon_i) = \left\langle \epsilon_i, \omega_{n+j} \right\rangle \neq 0$.  
\end{proof}

\begin{defn} \label[defn]{count}
Let $(n, j)$ be a pair of integers with $n\geq 3$ and $j\geq 0$.
Define $C(n, j)$ to be the set of odd primes satisfying that $p\geq j+2$, that $p\leq n+1$, and that the value of $n$ mod $p$ belongs to the set $\{0, p-1, \cdots, p-j\}$.
Write $\theta(n, j)$ for the size of the finite set $C(n, j)$.
\end{defn}

\begin{thm} \label{thm:counting}
There are at least $2^{\theta(n, j)}$ distinct integral homotopy types of gauge groups $\Gcal_{\sfP}$, as $\sfP$ ranges over the principal $\SU(n)$-bundles over $\Smr^{2n+2j+1}$.
\end{thm}

\begin{proof}
Let $p_1, p_2, \cdots, p_{\theta(n, j)}$ be a complete list of elements of $C(n, j)$ in increasing order.
For any $\theta(n, j)$-tuple $\vec{a} = (a_1, a_2, \cdots, a_{\theta(n, j)})$ where each $a_i$ equals $0$ or $1$, by the Chinese remainder theorem there is a unique integer $k$, $0\leq k < \prod_{1\leq i \leq \theta(n, j)} p_i$, so that $k \equiv a_i$ mod $p_i$ for $i=1, \cdots, \theta(n, j)$. Write $\Gcal_k$ for the gauge group of the principal $\SU(n)$-bundle over $\Smr^{2n+2j+1}$ classified by (an integral lift of) $k\omega_{n+j}$. The proof of \Cref{thm:appl} implies that, working \mbox{$p_i$-locally},
\[
(\partial_k)_*(\epsilon_i) = \left\langle \epsilon_i, k\omega_{n+j}  \right\rangle = k \left\langle \epsilon_i, \omega_{n+j}  \right\rangle = i! \, (n-1-(p_i-2-j)) \, k \, \omega_{n+N-(p-2)} \circ \alpha_1,
\]
which is nonzero mod $p_i$ precisely when $a_i = 1$. So if vectors $\vec{a}$ and $\vec{a}'$ differ in the $i$-th digit, then the corresponding gauge groups $\Gcal_k$ and $\Gcal_{k'}$ have different $p_i$-local homotopy types. Since there are $2^{\theta(n, j)}$ choices of the vector $\vec{a}$, we obtain $2^{\theta(n, j)}$ homotopy types of $\Gcal_k$.
\end{proof}

\smallskip

\begin{ex}
Consider the principal $\SU(5)$-bundles over $\Smr^{11}$, which are classified as
\[
\Prin_{\SU(5)}(\Smr^{11}) \cong \pi_{10}\SU(5) \cong \ZZ/5!.
\]
So up to isomorphism there are $5! = 120$ such bundles. We now apply \Cref{thm:counting} with 
$n=5$ and $j=0$. We find that
\[
C(5, 0) = \{5\}, \quad \text{and} \quad \theta(5, 0) = 1.
\]
So by \Cref{thm:counting} these $120$ bundles give rise to at least $2^1 = 2$ distinct homotopy types of gauge groups. In the next unstable case, the principal $\SU(4)$-bundles over $\Smr^{11}$ are classified as
\[
\Prin_{\SU(4)}(\Smr^{11}) \cong \pi_{10}\SU(4) \cong \ZZ/5! \oplus \ZZ/2.
\]
So up to isomorphism there are $240$ such bundles. With $n=4$ and $j=1$, we have
\[
C(4, 1) = \{5\}, \quad \text{and} \quad \theta(4, 1) = 1.
\]
By \Cref{thm:counting} these $240$ bundles give rise to at least $2^1=2$ distinct homotopy types of gauge groups. Note that \Cref{thm:counting} does not detect $\SU(3)$-gauge groups over $\Smr^{11}$, as $C(3, 2) = \emptyset$.
\end{ex}

\smallskip

In the following examples, we denote by $L_{n, j}$ the lower bound for the number of homotopy types  $\SU(n)$-gauge groups over $\Smr^{2n+2j+1}$, obtained by applying \Cref{thm:counting}.

\smallskip

\begin{ex}
Consider the $\SU(n)$-bundles over $\Smr^{31}$, in the unstable range $n\leq 15$. In these cases $n+j = 15$. Applying \Cref{thm:counting} gives the following results.

\smallskip

\begin{center}
\begin{tabular}{c||c|c|c|c|c|c|c}
\hline
$n$ &  $15$ & $14$ & $13$ & $12$ & $11$ & $10$ & $\leq 9$ \\
\hline
$C(n, j)$ &  $\{3, 5\}$ & $\{3, 5, 7\}$ & $\{5, 7, 13\}$ & $\{5, 7, 13\}$ & $\{7, 11\}$ & $\{7, 11\}$ & $\emptyset$ \\
\hline
$L_{n, j}$ &  $4$ & $8$ & $8$ & $8$ & $4$ & $4$ & $1$ \\
\hline
\end{tabular}
\end{center}
\end{ex}

\smallskip

\begin{ex}
Fix $n=15$ and consider the $\SU(15)$-bundles over $\Smr^{31+2j}$, in the unstable range $j\geq 0$.
Applying \Cref{thm:counting}, we obtain the following results.

\smallskip

\begin{center}
\begin{tabular}{c||c|c|c|c|c|c|c}
\hline
$j$ &  0, 1 & 2, 3 & 4-6 & 7-9 & 10 & 11 & $\geq 12$ \\
\hline
$C(15, j)$ &  $\{3, 5\}$ & $\{5\}$ & $\emptyset$ & $\{11\}$ & $\emptyset$ & $\{13\}$ & $\emptyset$ \\
\hline
$L_{15, j}$ &  $4$ & $2$ & $1$ & $2$ & $1$ & $2$ & $1$  \\
\hline
\end{tabular}
\end{center}
\end{ex}

\smallskip

\begin{ex} \label[ex]{corank1}
Consider the $\SU(n)$-bundles over $\Smr^{2n+1}$ for a general $n\geq 3$, classified as
\[
\Prin_{\SU(n)}(\Smr^{2n+1}) \cong \pi_{2n}\SU(n) \cong \ZZ/n!.
\]
So up to isomorphism there are $n!$ such bundles. Applying \Cref{thm:counting} with this general $n$ and $j=0$, we find that
\[
C(n, 0) = \{p: p \text{ is an odd prime dividing } n \}.
\]
So if $\omega(n)$ denotes the number of distinct prime divisors of $n$, then
\[
\theta(n, 0) = \begin{cases}
 \omega(n) & \text{if } n \text{ is odd} \\
 \omega(n) - 1 & \text{if } n \text{ is even}.
\end{cases}
\]
Therefore the number of homotopy types of $\SU(n)$-gauge groups over $\Smr^{2n+1}$ is at least
\[
 L(n, 0) = \begin{cases}
 2^{\omega(n)} & \text{if } n \text{ is odd} \\
 2^{\omega(n) - 1} & \text{if } n \text{ is even}.
\end{cases}
\]
\end{ex}

\smallskip

\begin{ex} \label[ex]{corank2}
Consider the $\SU(n)$-bundles over $\Smr^{2n+3}$ for a general $n\geq 3$, classified as
\[
\Prin_{\SU(n)}(\Smr^{2n+3}) \cong \pi_{2n+2}\SU(n) \cong
\begin{cases}
 \ZZ/(n+1)! \oplus \ZZ/2 & \text{if } n \text{ is even} \\
 \ZZ/\frac{(n+1)!}{2} & \text{if } n \text{ is odd}. 
\end{cases} 
\]
Applying \Cref{thm:counting} with this general $n$ and $j=1$, we find that
\[
C(n, 1) = \{p: p \text{ is an odd prime dividing } n  \text{ or } n+1\}.
\]
Note that $n$ and $n+1$ are necessarily coprime, so
\begin{align*}
\theta(n, 1) & = \#\{\text{odd primes dividing } n\} + \#\{\text{odd primes dividing } n+1\} \\
 & = \omega(n) + \omega(n+1) - 1.
\end{align*}
The last equality follows because exactly one of $n$ and $n+1$ is even.
Therefore the number of homotopy types of $\SU(n)$-gauge groups over $\Smr^{2n+3}$ is at least
\[
 L(n, 1) = 2^{\omega(n) + \omega(n+1) - 1}.
\]
\end{ex}

\smallskip


\bibliographystyle{amsalpha}

\bibliography{GG1}

@article {AB83,
    AUTHOR = {M.F. Atiyah and R. Bott},
     TITLE = {The {Y}ang-{M}ills equations over {R}iemann surfaces.},
   JOURNAL = {Philos. Trans. Roy. Soc. London Ser. A},
   NUMBER = {1505},
    VOLUME = {308},
      YEAR = {1983},
     PAGES = {523--615}
}

@article {CS00,
    AUTHOR = {M.C. Crabb and W.A. Sutherland},
     TITLE = {Counting homotopy types of gauge groups},
   JOURNAL = {Proc. London Math. Soc. (3)},
   NUMBER = {3},
    VOLUME = {81},
      YEAR = {2000},
     PAGES = {747--768}
}

@article {Go72,
    AUTHOR = {D.H. Gottlieb},
     TITLE = {Applications of bundle map theory},
   JOURNAL = {Trans. Amer. Math. Soc},
    VOLUME = {171},
      YEAR = {1972},
     PAGES = {23--50}
}

@article {Ha07,
    AUTHOR = {H. Hamanaka},
     TITLE = {On {S}amelson products in $p$-localized unitary groups},
   JOURNAL = {Topology Appl.},
   NUMBER = {3},
    VOLUME = {154},
      YEAR = {2007},
     PAGES = {573--583}
}

@article {HK03,
    AUTHOR = {H. Hamanaka and A. Kono},
     TITLE = {On {$[X, U(n)]$ when $\dim X$ is $2n$}},
   JOURNAL = {J. Math. Kyoto Univ.},
   NUMBER = {2},
    VOLUME = {43},
      YEAR = {2003},
     PAGES = {333--348}
}

@article {HK06,
    AUTHOR = {H. Hamanaka and A. Kono},
     TITLE = {Unstable {$K^1$}-group and homotopy type of certain gauge groups},
   JOURNAL = {Proc. Roy. Soc. Edinburgh Sect. A},
   NUMBER = {1},
    VOLUME = {136},
      YEAR = {2006},
     PAGES = {149--155}
}

@article {HK07,
    AUTHOR = {H. Hamanaka and A. Kono},
     TITLE = {Homotopy type of gauge groups of {$SU(3)$}-bundles over {$S^6$}},
   JOURNAL = {Topology Appl.},
   NUMBER = {7},
    VOLUME = {154},
      YEAR = {2007},
     PAGES = {1377--1380}
}

@article {Hu23,
    AUTHOR = {Y. Hu},
     TITLE = {Metastable complex vector bundles over complex projective spaces},
   JOURNAL = {Trans. Amer. Math. Soc.},
   NUMBER = {11},
    VOLUME = {276},
      YEAR = {2023},
     PAGES = {7783--7814}
}

@article {La73,
    AUTHOR = {G.E. Lang},
     TITLE = {The evaluation map and {EHP} sequences},
   JOURNAL = {Pacific J. Math.},
    VOLUME = {44},
      YEAR = {1973},
     PAGES = {201--210}
}

@article {Moh22,
    AUTHOR = {S. Mohammadi},
     TITLE = {The homotopy types of {$SU(n)$}-gauge groups over {$S^{2m}$}},
   JOURNAL = {Homol. Homot. Appl.},
   NUMBER = {1},
    VOLUME = {24},
      YEAR = {2022},
     PAGES = {55-70}
}

@article {MAG19,
    AUTHOR = {S. Mohammadi and M. Asadi-Golmankhaneh},
     TITLE = {The homotopy types of {$SU(4)$}-gauge groups over {$S^8$}},
   JOURNAL = {Topology Appl.},
    VOLUME = {266},
      YEAR = {2019},
     PAGES = {106845}
}

@article {So19,
    AUTHOR = {T. So},
     TITLE = {Homotopy types of gauge groups over non-simply connected closed 4-manifolds},
   JOURNAL = {Glasgow Math. J.},
    VOLUME = {61},
      YEAR = {2019},
     PAGES = {349--371}
}

\end{document}